\documentclass[11pt]{amsart}

\usepackage{hyperref}
\usepackage[T1]{fontenc}
\usepackage[utf8]{inputenc}

\usepackage{amssymb}

\makeatletter
\@namedef{subjclassname@2020}{%
  \textup{2020} Mathematics Subject Classification}
\makeatother

\allowdisplaybreaks[2]

\newtheorem{theorem}{Theorem}[section]

\newtheorem{corollary}[theorem]{Corollary}
\theoremstyle{definition}
\newtheorem{remark}[theorem]{Remark}

\newtheorem{example}[theorem]{Example}

\numberwithin{equation}{section}

\begin{document}

\title[Jordan maps and zLpd algebras]{J\lowercase{ordan maps and zero} L\lowercase{ie product determined algebras}}

\thanks{Supported by the Slovenian Research Agency (ARRS) Grant P1-0288. }

\author{M\lowercase{atej} Bre\v sar \url{https://orcid.org/0000-0001-7574-212X}}
\address{Faculty of Mathematics and Physics,  University of Ljubljana,  and Faculty of Natural Sciences and Mathematics, University of Maribor, Slovenia}
\email{matej.bresar@fmf.uni-lj.si}



\begin{abstract} Let $A$ be  an algebra over a field $F$ with {\rm char}$(F)\ne 2$. If $A$ is generated as an algebra by $[[A,A],[A,A]]$, then for every
 skew-symmetric bilinear map $\Phi:A\times A\to X$, where $X$ is an arbitrary vector space over $F$, the condition that $\Phi(x^2,x)=0 $
for all $x\in A$ implies that
$\Phi(xy,z) +\Phi(zx,y) + \Phi(yz,x)=0$ for all $x,y,z\in A$. This is applicable to the question of whether $A$ is zero Lie product determined, and is also used in proving that a Jordan homomorphism from $A$ onto a semiprime algebra $B$ is the sum of a homomorphism and an antihomomorphism.
\end{abstract}

\maketitle
\centerline{
{\em Dedicated to Vesselin Drensky on his 70th birthday}}

\newcommand\E{\ell}
\newcommand\mathcalM{{\mathcal M}}
\newcommand\pc{\mathfrak{c}}

\newcommand{\enp}{\begin{flushright} $\Box$ \end{flushright}}

\section{Introduction} 

This paper is centered around the following question. Given an algebra $A$, does every bilinear map $\Phi:A\times A\to X$, where $X$ is an arbitrary vector space, which satisfies
\begin{equation} \label{prv} \Phi(x^2,x)=0\end{equation} for all $x\in A$  also
satisfies
\begin{equation} \label{dru}\Phi(xy,z) +\Phi(zx,y) + \Phi(yz,x)=0\end{equation} for all $x,y,z\in A$? This question is intimately connected with zero Lie product determined algebras  and 
is also related to some problems that can be solved by using functional identities. 
 Let us explain this more precisely.

An algebra $A$ over a field $F$ is said to be {\em zero Lie product determined} (zLpd for short) if, for every bilinear map $\Phi:A\times A\to X$ where $X$ is a vector space over $F$, the condition that $\Phi(x,y)=0$ whenever $[x,y]=0$ implies that there exists a linear map $\tau:A\to X$ such that \begin{equation} \label{tre} \Phi(x,y) = \tau([x,y])\end{equation}
for all $x,y\in A$; here, $[x,y]$ of course stands for $xy-yx$.  These algebras are one of the topics of the recently published book  \cite{zpdbook}  (it should be remarked that in the above definition it is enough to require that $X$ is 1-dimensional, see \cite[Proposition 1.3]{zpdbook}). Note that
 $\Phi(x,y)=0$ whenever $[x,y]=0$  implies that $\Phi$  satisfies \eqref{prv} and is skew-symmetric (the latter follows from   $\Phi(x,x)=0$ for every $x$). 
  On the other hand,  \eqref{tre}  implies \eqref{dru}.  
The converse implication is not true.  
However, 
when applying the condition that an algebra is zLpd it is sometimes enough to use the fact that the corresponding bilinear map satisfies
\eqref{dru}. This has been the case in  the recent paper \cite{Bl} (and also \cite[Section 7.2]{zpdbook}) where the techniques of functional identities \cite{FIbook} have been applied to  the problem of describing commutativity preserving linear maps in zLpd algebras. It seems plausible that a similar approach can be used for some other problems to which functional identities  are applicable.

Throughout, we assume that char$(F) \ne 2$.
Our basic result, Theorem \ref{mt}, from which all the others follow in particular shows that \eqref{prv} implies \eqref{dru} provided that  $A$ is equal to $ {\rm Alg}([[A,A],[A,A]])$, the subalgebra generated by all elements of the form $[[x,y],[z,w]]$, $x,y,z,w\in A$. This is fulfilled for every noncommutative simple algebra $A$ (Corollary \ref{ca}), which reduces the question of whether a simple algebra $A$ is zLpd to the question of whether every derivation from $A$ to its dual $A^*$ is inner (Corollary \ref{cb}) as well as to the question of triviality of the second cohomology group $H ^2(A,F)$ (Corollary \ref{cb2}). In Section \ref{s3}, we apply Theorem \ref{mt} to show that every surjective Jordan homomorphism from an algebra $A$ satisfying $A= {\rm Alg}([[A,A],[A,A]])$ to a semiprime algebra $B$ is the sum of a homomorphism and an antihomomorphism (Theorem \ref{Jjorhom2}).

The proof of Theorem \ref{mt} is based on the idea that a bilinear map 
satisfying \eqref{prv} gives rise to a Jordan derivation from $A$ to its dual $A^*$ (which makes it possible for us to use a general result on Jordan derivations from \cite{Bj}). This idea is known from the analytic study of zLpd algebras, see \cite[Section 6.2]{zpdbook}.  In the course of writing the book \cite{zpdbook}, the author somehow overlooked that it  is applicable also in the algebraic context. The present paper may be viewed as a supplement to the book.

\section{Main results}
\label{s2}
Let $A$ be an algebra over a field $F$ and $M$ be an $A$-bimodule. Recall that a linear map $\delta:A\to M$ is called a {\em Jordan derivation} if 
\begin{equation} \label{jd} \delta(yz+zy)=\delta(y)\cdot z + y\cdot \delta(z) +  \delta(z)\cdot y  + z\cdot\delta(y)\end{equation}
for all $y,z\in A$. If $\delta$ satisfies $$\delta(yz)=\delta(y)\cdot z + y\cdot \delta(z) $$ for all $y,z\in A$, then it is called a {\em derivation}. Obviously, a derivation is also a Jordan derivation. The converse is not always true. However,
\cite[Theorem 3.1]{Bj} states that if  char$(F)\ne 2$, then every Jordan derivation $\delta:A\to M$ satisfies 
\begin{equation} \label{jd2}\delta(yu)=\delta(y)\cdot u+y\cdot \delta(u)\end{equation} for all  $u\in {\rm Alg}([[A,A],[A,A]])$ and all  $y\in A$.
This result is the key to our basic theorem, which we now state.

\begin{theorem} \label{mt}
Let $A$ be an algebra over a field $F$ with {\rm char}$(F)\ne 2$. If a skew-symmetric bilinear map $\Phi:A\times A\to X$, where $X$ is an arbitrary vector space over $F$, satisfies $$\Phi(x^2,x)=0 $$
for all $x\in A$, then 
$$\Phi(xy,u) +\Phi(ux,y) + \Phi(yu,x)=0$$ for all $u\in {\rm Alg}([[A,A],[A,A]])$ and all  $x,y\in A$.
\end{theorem}

\begin{proof}
By considering $\omega\circ\Phi$, where $\omega$ is an arbitrary 
linear functional on $X$, we see that 
we may assume that $X=F$.


Note that the dual space $A^*$ of $A$ becomes an $A$-bimodule by setting
$$(x\cdot f)(y) = f(yx)\quad\mbox{and}\quad (f\cdot x)(y) = f(xy)$$
for all $x,y\in A$, $f\in A^*$. Define $\delta:A\to A^*$ by
$$\delta(y)(x)=\Phi(y,x).$$
We claim that $\delta$ is a Jordan derivation. Indeed, 
linearizing  $\Phi(x^2,x)=0$ we obtain
$$ \Phi(xy+yx,z) + \Phi(zx+xz,y) + \Phi(yz + zy,x) =0$$
for all $x,y,z\in A$.
Since $\Phi$ is skew-symmetric, this can be rewritten as   
\begin{equation}\label{ll}\Phi(yz + zy,x)=  \Phi(z,xy+yx)+ \Phi(y,zx+xz).\end{equation}
One easily checks that \eqref{ll} implies  \eqref{jd}, which proves our  claim.

By \cite[Theorem 3.1]{Bj}, $\delta$ satisfies \eqref{jd2}
 for all  $u\in {\rm Alg}([[A,A],[A,A]])$ and all  $y\in A$, which
 yields the desired conclusion. 
\end{proof}

We continue with an  example showing  that \eqref{prv} does not always imply \eqref{dru}, so the involvement of  ${\rm Alg}([[A,A],[A,A]])$ in  Theorem \ref{mt} is necessary. Note that the algebra $A$ in this example satisfies $[[A,A],[A,A]]=\{0\}$. 

\begin{example} Let $F$ be any field and 
let $A$ be the unital $F$-algebra with generators $x_1,x_2,y_1,y_2$ satisfying the relations 
$$x_1x_2 + x_2x_1= y_1y_2 + y_2y_1 = x_i^2 = y_i^2 =x_iy_j = y_jx_i =0$$
for all $i,j=1,2$. That is, $A$ is the direct product of two copies of the Grassmann algebra in two generators without unity, to which we adjoin a unity. Observe that the elements $1,x_1, x_2, x_1x_2, y_1,y_2, y_1y_2$ form a basis of $A$, which is thus a $7$-dimensional algebra. Every element in the linear span  of $x_1, x_2, x_1x_2, y_1,y_2, y_1y_2$ has square $0$, so every bilinear functional $\Phi:A\times A\to F$ that satisfies $\Phi(x,x)=\Phi(1,x)=0$ for every $x\in A$ 
also satisfies  $\Phi(x^2,x)=0$ for every $x\in A$. 
Therefore, if $\Phi$ is such that $$\Phi(x_1x_2,y_1)=-\Phi(y_1,x_1x_2)=1$$
and $\Phi(u,v)=0$ for all other elements $u,v$ from our basis, then $\Phi$ is skew-symmetric and satisfies $\Phi(x^2,x)=0$ for every $x\in A$. However, for $x=x_1$, $y=x_2$, and $z=y_1$ we have $$\Phi(xy,z) + \Phi(zx,y) + \Phi(yz,x)=1.$$ 
\end{example}

If $A$ is equal to $ {\rm Alg}([[A,A],[A,A]])$, then Theorem \ref{mt} gives a definitive conclusion. Many algebras satisfy this condition. 
In particular, the following corollary holds.

\begin{corollary} \label{ca}
Let $A$ be a  simple algebra over a field $F$ with {\rm char}$(F)\ne 2$. If a skew-symmetric bilinear map $\Phi:A\times A\to X$, where $X$ is an arbitrary vector space over $F$, satisfies $$\Phi(x^2,x)=0$$ for all $x\in A$, then 
$$\Phi(xy,z) +\Phi(zx,y) + \Phi(yz,x)=0$$ for all $x,y,z\in A$.
\end{corollary}

\begin{proof}
If $A$ is commutative, then the desired conclusion follows immediately from the linearized form \eqref{ll} of our condition $\Phi(x^2,x)=0$.
Thus, assume that $A$ is not commutative. We claim that then $A= {\rm Alg}([[A,A],[A,A]])$. Indeed, this follows from the classical results by Herstein. Observe first that the Jacobi identity implies that $L=[[A,A],[A,A]]$, the linear span of all elements of the form $[[x,y],[z,w]]$, $x,y,z,w\in A$, is a Lie ideal of $A$ (i.e., $[L,A]\subseteq L$). Therefore, \cite[Theorem 1.3]{Her1} tells us that either $L = [A,A]$ or $L$ is contained in the center $Z$ of $A$. In the former case, the desired conclusion 
$A= {\rm Alg}(L)$ follows from \cite[Corollary on p.\ 6]{Her1}. Assuming that $L\subseteq Z$, we may use  \cite[Lemma 1]{Her2} which states that if $T$ is any Lie ideal of $A$, then $[T,T]\subseteq Z$ implies $T\subseteq Z$. 
Thus, $L\subseteq Z$ yields $[A,A]\subseteq Z$ which in turn implies $A\subseteq Z$, a contradiction.
\end{proof}

The next example shows that the assumption that char$(F)\ne 2$ is necessary in  Corollary \ref{ca}.

\begin{example}
Let $F$ be a field with char$(F)= 2$ and let $A=F(X)$ be the rational function field over $F$. Note that 
 $x^2\in   F(X^2)$ for every $x\in A$, i.e., the set of all squares in $A$ is contained in the subfield of rational functionals in $X ^2$. If $\Phi:A\times A\to F$ is any skew-symmetric (= symmetric in this context) bilinear functional such that 
$$\Phi\big( F(X^2), A\big)=\{0\},\,\,\, \Phi(X ^3,X^3)=0,\,\,\,\mbox{and}\,\,\,\Phi(X^5,X)=1,$$ then  $\Phi$ satisfies $\Phi(x^2,x)=0$ for every $x\in A$, but $$\Phi(xy,z) + \Phi(zx,y) + \Phi(yz,x)=1$$ for $x=X$, $y=X^2$, and $z=X^3$.
\end{example}

We conclude this section with two results   giving  criteria  for  a simple algebra to be zLpd.   In the statement of the first one we consider the dual space $A^*$ as an $A$-bimodule (as in the proof of Theorem  \ref{mt}). Recall that a derivation $\delta$ from $A$ to a  bimodule $M$ is said to be {\em inner} if there exists a $\tau \in M$ such that $\delta(x)=\tau \cdot x - x\cdot \tau$ for all $x\in A$.

\begin{corollary}\label{cb}
Let $A$ be a simple algebra over a field $F$ with char$(F)\ne 2$. If every derivation from $A$ to $A^*$ is inner, then $A$ is zLpd.
\end{corollary}

\begin{proof} Let $\Phi:A\times A\to F$ be a bilinear functional such that $\Phi(x,y)=0$ whenever $[x,y]=0$. Then $\Phi$ is skew-symmetric and satisfies  $\Phi(x^2,x)=0$ for every $x\in A$. By Corollary \ref{ca}, $\Phi$ satisfies $\Phi(xy,z) =\Phi(y,zx) + \Phi(x,yz)$ for all $x,y,z\in A$, which means  that  $\delta:A\to A^*$ defined by $\delta(x)(z)=\Phi(x,z)$ is a derivation. By our assumption, there exists a $\tau\in A ^*$ such that $\delta(x) = \tau \cdot x - x\cdot \tau$ for all $x\in A$. Hence,
$$\Phi(x,y)= \delta(x)(y)= (\tau\cdot x - x\cdot\tau)(y)= \tau(xy - yx)$$
for all $x,y\in A$, which proves that $A$ is zLpd.
\end{proof}

\begin{remark}Let $A$ be a finite-dimensional simple algebra over a field $F$ with char$(F)\ne 2$. If $A$ is separable, then every derivation from $A$ to any $A$-bimodule $M$ is inner \cite[Theorem 8.1.19]{Ford}, so $A$ is zLpd by Corollary \ref{cb}. Furthermore, since
the direct product of finitely many zLpd algebras is again a zLpd algebra \cite[Theorem 1.16]{zpdbook}, it follows that every separable algebra (over a field $F$ with char$(F)\ne 2$) is zLpd. In particular, every finite-dimensional semisimple  algebra over a perfect field $F$ with char$(F)\ne 2$ is zLpd. This result is new, but we do not know whether it actually holds for finite-dimensional semisimple  algebras over arbitrary fields.
\end{remark}

Recall that a skew-symmetric bilinear functional  $\Phi:A\times A\to F$ is called a {\em 2-cocycle} if $$\Phi([x,y],z) +\Phi([z,x],y) + \Phi([y,z],x)=0$$ for all $x,y,z\in A$, and is called a {\em coboundary} if there exists a $\tau$ in $A^*$ such that
$\Phi(x,y)=\tau([x,y])$ for all $x,y\in A$. It is clear that every  coboundary is a 2-cocycle. The condition that the converse is also true (i.e., every 2-cocycle is a coboundary)
 can be stated as that the second cohomology group $H ^2(A,F)$ is trivial.

\begin{corollary}\label{cb2}
Let $A$ be a simple algebra over a field $F$ with char$(F)\ne 2$. If   the second cohomology group $H ^2(A,F)$ is trivial, then $A$ is zLpd.
\end{corollary}

\begin{proof}
As in the preceding proof we see that  a bilinear functional $\Phi$ vanishing on pairs of commuting elements satisfies
 $\Phi(xy,z) =\Phi(y,zx) + \Phi(x,yz)$, $x,y,z\in A$. This readily implies that $\Phi$ is a $2$-cocycle and hence  a coboundary by our assumption. This means that $A$ is zLpd. 
\end{proof}

Examples of algebras satisfying the conditions of Corollary \ref{cb2} can be found in \cite{Su2,Su,  Zhao}.  The author is thankful to Kaiming Zhao for pointing out these references.

Let us mention that Corollaries \ref{cb} and \ref{cb2} are quite similar, but written in a different mathematical language. Namely, the condition that every 2-cocycle of $A$ is a coboundary is equivalent to the condition that every Lie derivation  $\delta:A\to A ^*$ such that $\delta(x)(y)=-\delta(y)(x)$ for all $x,y\in A$ is inner.

\section{An application to Jordan homomorphisms}\label{s3}

Let $A$ and $B$ be $F$-algebras. A linear map $J:A\to B$ is called a {\em Jordan homomorphism} if $$J(yz + zy)=J(y)J(z)+ J(z)J(y)$$
for all $y,z\in A$. Obvious examples are homomorphisms and antihomomorphisms. A more general example is {\em 
the sum of a
homomorphism and an antihomomorphism}. By this we mean a linear map
$J:A\to B$  for which there exists ideals
$U$ and $V$ of Alg$(J(A))$ (the subalgebra of $B$ generated by the image of $J$), 
a homomorphism $H : A\to U$, and  an antihomomorphism $K : A \to V$ such that $J=H+K$ and $UV = VU =\{ 0\}$. 

The problem of expressing  Jordan homomorphisms $J$ through homomorphisms, antihomomorphisms, and their sums has a long history; see \cite{Bj2} for a brief survey.  One often assumes  that $J$ is surjective and the algebra $B$ is semiprime, i.e.,  it has no nonzero nilpotent ideals. If, under these assumptions, $J$ is the sum of a homomorphism and an antihomomorphism, then \cite[Lemma 2.1]{Bj2} states that there exist ideals $U_0$ and $V_0$ of $A$ and ideals $U$ and $V$ of $B$ such that
\begin{itemize}\item $U_0 + V_0 = A$ and $U_0\cap V_0 = {\rm ker} J$,
\item $U \oplus V = B$,  \item $\left.J\right|_{U_0}$ is a 
 homomorphism from $U_0$ onto $U$, \item $\left.J\right|_{V_0}$ is a 
 homomorphism from $V_0$ onto $V$.\end{itemize}
The description of $J$ is thus very clear in this case. However, not every Jordan homomorphism from an algebra $A$   onto a semiprime algebra $B$ is the sum of a homomorphism and an antihomomorphism, see \cite[p.\ 458]{BM} and \cite[Example 7.19]{zpdbook}. Our goal is to use Theorem \ref{mt} to show that this does hold if $A$ is equal to ${\rm Alg}([[A,A],[A,A]])$ (and of course char$(F)\ne 2$).

We now list other results that will be needed in the proof.
Assume that  $J:A\to B$ is a Jordan homomorphism, where $A$ and $B$ are   algebras over a field $F$ with 
{\rm char}$(F)\ne 2$. The following statements are known.
\begin{enumerate}
\item[(a)] $J$ satisfies
\begin{equation}\label{d}J(zyz) = J(z)J(y)J(z)\end{equation}
for all $y,z\in A$, and 
\begin{equation}\label{s}J(xyz + zyx) = J(x)J(y)J(z) + J(z)J(y)J(x)\end{equation}
for all $x,y,z\in A$. This is standard and easy to prove: \eqref{d} follows from 
$2zyz = \big(z (z y + yz )  + (zy+yz)z\big)- (z^2 y+yz^2)$, and \eqref{s} follows by linearizing \eqref{d}.
\item[(b)] If $J$ is surjective and $B$ is semiprime, then
\begin{equation}\label{p}\big(J(zw)- J(z)J(w)\big)\big( J(xy)-J(y)J(x)\big)=0\end{equation}
and
\begin{equation}\label{q}\big(J(wz) -J(z)J(w)\big)\big( J(yx)-J(y)J(x)\big)=0\end{equation}
for all $x,y,z,w\in A$. See \cite[Corollary 2.2]{Bjh}.
\item[(c)] 
If $J$ is a reversal  homomorphism, which means that
$$J(xyzw + wzyx) = J(x)J(y)J(z)J(w) + J(w)J(z)J(y)J(x)$$
for all $x,y,z,w\in A$,
 and Alg$(J(A))$
does not contain 
nonzero nilpotent ideals contained in the center of $B$, then  the restriction of $J$ to the commutator ideal
 of $A$ (i.e., the ideal generated by all commutators $[x,y]$ in $A$)  is the sum of a homomorphism and an antihomomorphism \cite[Theorem 4.2]{Bj2}.
\end{enumerate}

We now have enough information to prove the following theorem.

\begin{theorem} \label{Jjorhom2}
Let $A$ and $B$ be   algebras over a field $F$ with 
{\rm char}$(F)\ne 2$. Suppose that ${\rm Alg}([[A,A],[A,A]]) =A$ and that $B$ is semiprime. Then every surjective Jordan homomorphism  $J:A\to B$ is the sum of a homomorphism and an antihomomorphism. 
\end{theorem}

\begin{proof}
Define $\Phi:A\times A\to B$ by
$$\Phi(x,y)=J(xy)-J(x)J(y).$$
It is clear that $\Phi$ is skew-symmetric, and,
by \eqref{d}, $\Phi(x^2,x)=0$ for every $x\in A$.  Theorem \ref{mt} thus tells us that
$$\Phi(xy,z) +\Phi(zx,y) + \Phi(yz,x)=0$$ for all $x,y,z\in A$. That is,
$$J(xyz) + J(zxy + yzx)= J(xy)J(z) + J(zx)J(y) + J(yz)J(x).$$
Since 
$$J(zxy + yzx)= J((zx)y + y(zx)) = J(zx)J(y) +J(y)J(zx),$$
we can rewrite this  as
$$J(xyz) = J(xy)J(z) +  J(yz)J(x)-J(y)J(zx).$$
This gives
\begin{align*}
J(xyzw) = J((xy)zw)=
J(xyz)J(w) +J(zw)J(xy)-J(z)J(wxy)
\end{align*}
and
\begin{align*}
J(wzyx) = J(wz(yx))=
J(wz)J(yx) +J(zyx)J(w)-J(z)J(yxw).
\end{align*}
Hence,
\begin{align*}
&J(xyzw + wzyx)\\
 =& J(xyz+zyx)J(w) + J(zw)J(xy) + J(wz)J(yx)    - J(z)J(wxy + yxw),
\end{align*}
which together with \eqref{s} yields
\begin{align}\label{eskor}
&J(xyzw + wzyx) \nonumber\\
=&
J(zw)J(xy) + J(wz)J(yx) \\
&+ J(x) J(y) J(z) J(w) -
J(z) J(w) J(x) J(y). \nonumber
\end{align}
Next, by \eqref{p}, 
$$J(zw)J(xy)=J(z)J(w)J(xy) + J(zw)J(y)J(x) - J(z)J(w)J(y)J(x)$$
and  by \eqref{q},
$$J(wz)J(yx)= J(z)J(w)J(yx) + J(wz)J(y)J(x) - J(z)J(w)J(y)J(x).$$
Accordingly,
\begin{align*}
&J(zw)J(xy) + J(wz)J(xy)\\
=& J(z)J(w)J(xy+yx) + J(zw+wz)J(y)J(x)\\& - 2J(z)J(w)J(y)J(x)
\\
=& J(z)J(w)J(x)J(y) + J(z)J(w)J(y)J(x)  \\
&+J(z)J(w)J(y)J(x)
+ J(w)J(z)J(y)J(x)\\
 &- 2J(z)J(w)J(y)J(x)\\
 =& J(z)J(w)J(x)J(y) + J(w)J(z)J(y)J(x).
\end{align*}
From \eqref{eskor} it now follows that
\begin{align*}
&J(xyzw + wzyx) \\
= &J(z)J(w)J(x)J(y) + J(w)J(z)J(y)J(x) \\
 & + J(x) J(y) J(z) J(w) -
J(z) J(w) J(x) J(y)\\
=&J(x)J(y)J(z)J(w) + J(w) J(z) J(y) J(x).
\end{align*}
We have thus proved that $J$ is a reversal homomorphism.  Since $J$ is surjective and $B$ semiprime,  Alg$(J(A))$
does not contain 
nonzero nilpotent ideals, and since  $A={\rm Alg}([[A,A],[A,A]])$, $A$ is also equal to its commutator ideal. The desired conclusion therefore follows from 
(c).
\end{proof}

It 
is quite common that results on homomorphisms are applicable to the study of derivations. In the proof of Theorem \ref{Jjorhom2}, a result on (Jordan) derivations was (indirectly, via the proof of Theorem \ref{mt}) used to obtain a result on (Jordan) homomorphisms. This, however, seems rather unusual and surprising.

Finally, we remark that the computations in the proof of Theorem \ref{Jjorhom2} are similar to those in the proof of \cite[Theorem 7.22]{zpdbook}. However, while this latter result also considers Jordan homomorphisms onto semiprime algebras, the main idea of its proof is different (its main assumption is that the algebra $A$ is zero product determined).

\end{document}